\documentclass{scrartcl}
\title{Secondary Characteristic Classes of \\ Lie Algebra Extensions}
\author{Stefan Wagner}
\date{}

\usepackage{fixme}

\KOMAoptions{parskip=half, paper=a4, abstract=on}
\usepackage[latin1]{inputenc}
\usepackage{fixltx2e}
\usepackage{lmodern}
\usepackage{microtype}		
\usepackage[english]{babel}
\usepackage{hyperref}
\usepackage{amsmath, amsthm, amssymb, bbm, mathtools, nicefrac}
\usepackage{enumitem}

	\newlist{equivalence}{enumerate}{1}
	\setlist[equivalence]{label=\textnormal{(}\alph*\textnormal{)}}
\usepackage{xspace}
\usepackage{xifthen}

\theoremstyle{plain}
	\newtheorem{thm}{Theorem}[section]
	
	\newtheorem{lemma}[thm]{Lemma}
	\newtheorem{cor}[thm]{Corollary}

\theoremstyle{definition}
	\newtheorem{defn}[thm]{Definition}
	
	\newtheorem{expl}[thm]{Example}
\theoremstyle{remark}

\newcommand*{\R}{\mathbb R}		
	
\DeclareMathOperator{\ev}{ev}

\DeclareMathOperator{\Hom}{Hom}

\DeclareMathOperator{\End}{End}

\DeclareMathOperator{\ad}{ad}

\DeclareMathOperator{\sgn}{sgn}

\DeclareMathOperator{\Alt}{Alt}
\DeclareMathOperator{\Sym}{Sym}

\DeclareMathOperator{\der}{der}

\newcommand{\cf}{\mbox{cf.}\xspace}			
\newcommand{\eg}{\mbox{e.\,g.}\xspace}			
\newcommand*{\ie}{\mbox{i.\,e.}\xspace}			

\DeclarePairedDelimiterX{\lprod}[2]{\, \prescript{}{#1}{\langle}}{\rangle}{#2}
\DeclarePairedDelimiterX{\rprod}[2]{\langle}{\rangle_{#1}}{#2}

\begin{document}

\author{
	Stefan Wagner \thanks{
		Blekinge Tekniska H\"ogskola,
		\href{mailto:stefan.wagner@bth.se}{\nolinkurl{stefan.wagner@bth.se}}
	}
}
\sloppy
\maketitle

\begin{abstract}
	\noindent
	We introduce a notion of secondary characteristic classes of Lie algebra extensions. As a spin-off of our construction we obtain a new proof of Lecomte's generalization of the Chern--Weil homomorphism.

	\vspace*{0,5cm}


\end{abstract}

\section{Introduction}
Characteristic classes are topological invariants of principal bundles and vector bundles associated to principal bundles. The theory of characteristic classes was started in the 1930s by Stiefel and Whitney. Stiefel studied certain homology classes of the tangent bundle $TM$ of a smooth manifold $M$, while Whitney considered the case of an arbitrary sphere bundle and introduced the concept of a characteristic cohomology class. In the next decade, Pontryagin constructed important new characteristic classes by studying the homology of real Grassman manifolds and Chern defined characteristic classes for complex vector bundles. Nowadays, characteristic classes are an important tool for many disciplines such as analysis, geometry and modern physics. For example, they provide a way to measure the non-triviality of a principal bundle respectively the non-triviality of an associated vector bundle. The Chern--Weil homomorphism of a principal bundle is an algebra homomorphism from the algebra of polynomials invariant under the adjoint action of a Lie group $G$ on the corresponding Lie algebra $\mathfrak{g}$, into the even de Rham cohomology $H_{\text{dR}}^{2\bullet}(M,\mathbb{K})$ of the base space $M$ of a principal bundle $P$ with structure group $G$. This map is achieved by evaluating an invariant polynom $f$ of degree $k$ on the curvature $\Omega$ of a connection $\omega$ on $P$ and thus obtaining a closed form on the base. 

In \cite{Lec85} Lecomte described a cohomological construction which generalizes the classical Chern--Weil homomorphism. In fact, Lecomte's construction associates characteristic classes to every Lie algebra extension and the classical construction of Chern and Weil arises in this context from the Lie algebra extension known as the \emph{Atiyah sequence} of a principal bundle.

In the 1970s, another set of characteristic classes called the \emph{secondary characteristic classes} has been discovered. These classes are also global invariants of principal bundles and are derived in a similar way as the primary characteristic classes from the curvature of adequate connection 1-forms. Secondary characteristic classes appear for example in the Lagrangian formulation of modern quantum field theories and most well-known might be the so-called \emph{Chern--Simons classes} (see \eg \cite{CS74,AzIz95,MaMa92} and references therein).

In this note we investigate secondary characteristic classes in the context of Lie algebra extensions which will also yield a new proof of Lecomte's generalization of the Chern--Weil homomorphism. Moreover, we would like to point out that our construction may be used to associate characteristic classes to split Lie algebra extensions, that is, to semidirect products of Lie algebras, because in this situation the primary characteristic classes vanish. From this perspective our results nicely complement the Lie algebraic theory of characteristic classes. More detailedly, the paper is organized as follows.

After this introduction and some preliminaries, we study in Section \ref{sec:cov+cur} the geometric notions of covariant derivative and curvature from an Lie algebraic point of view. 
In Section \ref{sec:lec} we pave the way for our main result and we give a short overview over Lecomte's generalization of the Chern--Weil homomorphism. Section  \ref{sec:main} is finally devoted to secondary characteristic classes of Lie algebra extensions. To be more precise, we first introduce the so-called \emph{Bott--Lecomte} homomorphism and show that this map, in fact, gives rise to classes in Lie algebra cohomology (Theorem~\ref{thm:main thm}). As a spin-off of our construction we obtain a new proof of Lecomte's generalization of the Chern--Weil homomorphism (Corollary~\ref{cor1} and Corollary~\ref{cor2}). Last but not least we introduce a notion of secondary characteristic classes of Lie algebra extensions and conclude with an example involving the oscillator algebra.

\section{Preliminaries and Notations}\label{sec:pre+not}

In this section we provide the most important definitions and notations of Lie algebra cohomology which are repeatedly used in this article. For a detailed background on Lie algebra cohomology we refer, for example, to \cite{ChEi48,AzIz95}.

\begin{enumerate}
	\item
		Let $V,W$ be vector spaces and $p\in\mathbb{N}_0$. We call a $p$-linear map $f:W^p\rightarrow V$ \emph{alternating} if $f(w_{\sigma(1)},\ldots,w_{\sigma(p)})=\sgn(\sigma)\cdot f(w_1,\ldots,w_p)$ for all $w_1,\ldots,w_p \in W$ and all permutations $\sigma\in S_p$. Furthermore, we write $\Alt^p(W,V)$ for the space of alternating $p$-linear maps $W^p\rightarrow V$. Given a $p$-linear map $f:W^p\rightarrow V$, it is easily checked that 
\begin{align*}
\Alt(f):=\sum_{\sigma\in S_p}\sgn(\sigma)\cdot f^{\sigma} 
\end{align*}
defines an element in $\Alt^p(W,V)$, where $f^{\sigma}(w_1,\ldots,w_p):=f(w_{\sigma(1)},\ldots,w_{\sigma(p)})$. Of particular interest are also the \emph{symmetric} $p$-linear maps $W^p\rightarrow V$ which are defined accordingly and denoted by $\Sym^p(W,V)$. We will sometimes consider an element $f \in \Sym^p(W,V)$ as a linear map defined on the \emph{symmetric tensor product} $S^p(W)$, that is, as a map $\tilde{f}:S^p(W)\rightarrow V$ satisfying $\tilde{f}(w_1\otimes_s\cdots\otimes_s w_p)=f(w_1,\ldots,w_p)$  for all $w_1,\ldots,w_p \in W$.

	\item
		Suppose that $\mathfrak{g}$ and $V_i$, $i=1,2,3$, are vector spaces and that $m:V_1\times V_2\rightarrow V_3$, $(v_1,v_2)\mapsto v_1\cdot_{m}v_2$ is a bilinear map. For $\alpha\in\Alt^p(\mathfrak{g},V_1)$ and $\beta\in\Alt^q(\mathfrak{g},V_2)$ we define the \emph{wedge product} $\alpha\wedge_m\beta\in\Alt^{p+q}(\mathfrak{g},V_3)$ by putting
\begin{align*}
\alpha\wedge_m\beta:=\frac{1}{p!q!}\Alt(\alpha\cdot_m\beta), 
\end{align*}
where $(\alpha\cdot_m\beta)(x_1,\ldots,x_{p+q}):=\alpha(x_1,\ldots,x_p)\cdot_m\beta(x_{p+1},\ldots,x_{p+q})$. The wedge products $\wedge_{\otimes_s}$ induced by the canonical multiplications $S^p(\mathfrak{g})\times S^q(\mathfrak{g})\rightarrow S^{p+q}(\mathfrak{g})$, $(x,y)\mapsto x\otimes_s y$ will be of particular interest to us.

	\item
		Let $\mathfrak{g}$ be a Lie algebra, $V$ a $\mathfrak{g}$-module and $p\in\mathbb{N}_0$. We denote the space of alternating $p$-linear mappings $\mathfrak{g}^p\rightarrow V$ by $C^p(\mathfrak{g},V):=\Alt^p(\mathfrak{g},V)$ and call its elements $p$-\emph{cochains}. 
On each $C^p(\mathfrak{g},V)$ we define the \emph{Chevalley-Eilenberg differential} $d_{\mathfrak{g}}^p=d_{\mathfrak{g}}$ by
\begin{align}
d_{\mathfrak{g}}\omega(x_0,\ldots,x_p)&:=\sum^p_{j=0}(-1)^jx_j.\omega(x_0,\ldots,\widehat{x}_j,\ldots,x_p)\notag\\
&+\sum_{i<j}(-1)^{i+j}\omega([x_i,x_j],x_0,\ldots,\widehat{x}_i,\ldots,\widehat{x}_j,\ldots,x_p),\notag
\end{align}
where $\widehat{x}_j$ means that $x_j$ is omitted. Observe that the right hand side defines for each $\omega\in C^p(\mathfrak{g},V)$ an element of $C^{p+1}(\mathfrak{g},V)$ because it is alternating. Putting all differentials together, we obtain a linear map $d_{\mathfrak{g}}:C(\mathfrak{g},V)\rightarrow C(\mathfrak{g},V)$. The elements of the subspace $Z^p(\mathfrak{g},V):=\ker({d_{\mathfrak{g}}}_{\mid C^p(\mathfrak{g},V)})$
are called $p$-\emph{cocycles}, and the elements of the spaces
\begin{align*}
B^p(\mathfrak{g},V):=d_{\mathfrak{g}}(C^{p-1}(\mathfrak{g},V)) \quad \text{and} \quad B^0(\mathfrak{g},V):=\{0\}
\end{align*}
are called $p$-\emph{coboundaries}. It can be shown that $d_{\mathfrak{g}}^2=0$, which implies that $B^p(\mathfrak{g},V)\subseteq Z^p(\mathfrak{g},V)$, so that it makes sense to define the $p^{\text{th}}$-\emph{cohomology space of} $\mathfrak{g}$ \emph{with values in the module} $V$, that is, $H^p(\mathfrak{g},V):=Z^p(\mathfrak{g},V)/B^p(\mathfrak{g},V)$.

	\item
		Let  $\mathfrak{g}$ and $V$ be Lie algebras. Furthermore, consider $V$ as a trivial $\mathfrak{g}$-module. Then the \emph{curvature} $R_{\sigma} \in C^2(\mathfrak{g},V)$ of an element $\sigma\in C^1(\mathfrak{g},V)$ is defined by
\begin{align*}
R_{\sigma}:=d_{\mathfrak{g}}\sigma+\frac{1}{2}[\sigma,\sigma], \quad \ie, \quad R_{\sigma}(x,y)=[\sigma(x),\sigma(y)]-\sigma([x,y]) \qquad \forall x,y \in \mathfrak{g},
\end{align*}
		
\end{enumerate}

\section{Covariant Derivatives and Curvature}\label{sec:cov+cur}

The purpose of this section is to study the geometric notions of covariant derivative and curvature from an Lie algebraic point of view. Indeed, let $\mathfrak{g}$ be a Lie algebra and $V$ a vector space, considered as a trivial $\mathfrak{g}$-module. We first twist the corresponding Chevalley--Eilenberg complex $(C^\bullet(\mathfrak{g},V),d_{\mathfrak{g}})$ with a linear map $S:\mathfrak{g}\rightarrow\End(V)$ and note that the bilinear evaluation map $\ev:\End(V)\times V\rightarrow V$ yields a linear operator
\begin{align*}
S_{\wedge}:C^p(\mathfrak{g},V)\rightarrow C^{p+1}(\mathfrak{g},V), \quad \alpha\mapsto S\wedge_{\ev}\alpha.
\end{align*}
The corresponding \emph{covariant derivative on} $C(\mathfrak{g},V)$ is then defined by
\begin{align*}
d_S:=S_{\wedge}+d_{\mathfrak{g}}:C^p(\mathfrak{g},V)\rightarrow C^{p+1}(\mathfrak{g},V) \qquad p\in\mathbb{N}_0.
\end{align*}
A few moments thought show that this can also be written as
\begin{align*}
d_S\alpha(x_0,\ldots,x_p)&:=\sum^p_{j=0}(-1)^j S(x_j).\alpha(x_0,\ldots,\widehat{x}_j,\ldots,x_p)\\
&+\sum_{i<j}(-1)^{i+j}\alpha([x_i,x_j],x_0,\ldots,\widehat{x}_i,\ldots,\widehat{x}_j,\ldots,x_p).
\end{align*}


The following result shows that the Lie algebraic notion of curvature satisfies a generalized Bianchi identity.

\begin{lemma}\label{lem:cur+bianchi}
Let $\sigma\in C^1(\mathfrak{g},V)$ and $S:=\ad\circ\sigma$. Then the corresponding curvature $R_{\sigma} \in C^2(\mathfrak{g},V)$ satisfies the abstract Bianchi identity $d_S R_{\sigma}=0$, that is,
\begin{align*}
\sum_{\text{cyc.}}[\sigma(x),R_{\sigma}(y,z)]-R_{\sigma}([x,y],z)=0 \qquad \forall x,y,z \in \mathfrak{g}.
\end{align*}
\end{lemma}
\begin{proof}
First note that $d_{\mathfrak{g}}[\sigma,\sigma]=[d_{\mathfrak{g}}\sigma,\sigma]-[\sigma,d_{\mathfrak{g}}\sigma]=2[d_{\mathfrak{g}}\sigma,\sigma]$. Therefore, a few moments thought show that
\begin{align*}
d_SR_{\sigma}&=(d_{\mathfrak{g}}+S_{\wedge})R_{\sigma}=d_{\mathfrak{g}}^2\sigma+\frac{1}{2}d_{\mathfrak{g}}[\sigma,\sigma]+S\wedge R_{\sigma}=[d_{\mathfrak{g}}\sigma,\sigma]+[\sigma,R_{\sigma}]
\\
&=[d_{\mathfrak{g}}\sigma,\sigma]-[R_{\sigma},\sigma]=-\frac{1}{2}[[\sigma,\sigma],\sigma]=0.
\end{align*}
\end{proof}


\begin{lemma}\label{lem:cov der}
Given a bilinear map $m:V\times V\rightarrow V$ and a linear map $S:\mathfrak{g}\rightarrow\der(V,m)$, we have for $\alpha\in C^p(\mathfrak{g},V)$ and $\beta\in C^q(\mathfrak{g},V)$ the relation
\begin{align*}
d_S(\alpha\wedge_m\beta)=d_S\alpha\wedge_m\beta+(-1)^p\alpha\wedge_m d_S\beta.
\end{align*}
\end{lemma}
\begin{proof}
Since $d_S=d_{\mathfrak{g}}+S_{\wedge}$, the assertion is easily verified for $S=0$. It therefore remains to show that $S\wedge(\alpha\wedge_m\beta)=(S\wedge\alpha)\wedge_m\beta+(-1)^p\alpha\wedge_m(S\wedge\beta)$. For this first recall that
\begin{align*}
S\wedge(\alpha\wedge_m\beta)=\frac{1}{p!q!}\Alt(S\cdot(\alpha\cdot_m\beta))
\end{align*}
and note further that $S(\mathfrak{g})\subseteq\der(V,m)$ implies $S\cdot(\alpha\cdot_m\beta))=(S\cdot\alpha)\cdot_m\beta+(\alpha\cdot_m(S\cdot\beta))^{\sigma}$, where $\sigma=(1 2\ldots p+1)\in S_{p+q+1}$ is a cycle of length $p+1$. Hence, the claim is a consequence of $\sgn(\sigma)=(-1)^p$.
\end{proof}

\section{Lecomte's Generalization of the Chern--Weil map}\label{sec:lec}

In this short section we pave the way for our main result and we give a short overview over Lecomte's generalization of the Chern--Weil homomorphism. In particular, at this point Lie algebra extensions enter the picture.
Indeed, let
\begin{align*}
0 \longrightarrow \mathfrak{n} {\longrightarrow}\widehat{\mathfrak{g}} \stackrel{q} {\longrightarrow} \mathfrak{g}\longrightarrow 0
\end{align*}
be a Lie algebra extension and $V$ a $\mathfrak{g}$-module which we also consider as a $\widehat{\mathfrak{g}}$-module with respect to the action $x.v:=q(x).v$ for $x \in \widehat{\mathfrak{g}}$ and $v \in V$. Furthermore, let $\sigma:\mathfrak{g}\rightarrow\widehat{\mathfrak{g}}$ be a linear section of $q$ and $R_{\sigma} \in C^2(\mathfrak{g},\mathfrak{n})$ the corresponding curvature. Given $p \in \mathbb{N}_0$ and $f \in \Sym^p(\mathfrak{n},V)$, we put
\begin{align*}
f_{\sigma}=\tilde{f}\circ(R_{\sigma}\wedge_{\otimes_s}\cdots\wedge_{\otimes_s}R_{\sigma}) \in C^{2p}(\mathfrak{g},V)
\end{align*}
(\cf Section \ref{sec:pre+not}). Moreover, we write $\Sym^p(\mathfrak{n},V)^{\widehat{\mathfrak{g}}}$ for the set of $\widehat{\mathfrak{g}}$-invariant symmetric $p$-linear maps, that is, the set
\begin{align*}
&{\Big\{}f\in\Sym^p(\mathfrak{n},V) \,|\, x.f(y_1,\ldots,y_p)=\sum^p_{i=1}
f(y_1,\ldots,S(x)y_i,\ldots,y_p) \qquad \forall x \in \mathfrak{g}\Big{\}},
\end{align*}
where $S:\mathfrak{g} \rightarrow \der(\mathfrak{n})$ denotes the linear map defined by $S(x):=\ad(\sigma(x))$. Lecomte's generalization of the Chern--Weil map then reads as follows:

\begin{thm}\label{thm:lecomte}
\emph{(Lecomte \cite[Thm. 2.3]{Lec85})}
\begin{itemize}
\item[\emph{(a)}]
For each $p\in\mathbb{N}_0$, there is a natural map
\begin{align*}
C_p:\Sym^p(\mathfrak{n},V)^{\widehat{\mathfrak{g}}}\rightarrow H^{2p}(\mathfrak{g},V), \quad f\mapsto\frac{1}{p!}[f_{\sigma}],
\end{align*}
which is independent of the choice of the section $\sigma$.
\item[\emph{(b)}]
Suppose, in addition, that $m_V:V\times V\rightarrow V$ is an associative multiplication and that $\mathfrak{g}$ acts on $V$ by derivations, i.e., $m_V$ is $\mathfrak{g}$-invariant. Then $(C^{\bullet}(\mathfrak{g},V),\wedge_{m_V})$ is an associative algebra, inducing an algebra structure on $H^{\bullet}(\mathfrak{g},V)$. Further, $(\Sym^{\bullet}(\mathfrak{n},V),\vee_{m_V})$ is an associative algebra, and the maps $(C_p)_{p\in\mathbb{N}_0}$ combine to an algebra homomorphism
\begin{align*}
C:\Sym^{\bullet}(\mathfrak{n},V)^{\widehat{\mathfrak{g}}}\rightarrow H^{2\bullet}(\mathfrak{g},V).
\end{align*}
\end{itemize}
\end{thm}

\section{Secondary Characteristic Classes of Lie algebra Extensions}\label{sec:main}

In this section we finally introduce a notion of secondary characteristic classes of Lie algebra extensions. As a spin-off of our construction we will obtain another proof for Lecomte's generalization of the Chern--Weil homomorphism. 

In the following we write $D_n$ for the vertical projection of the standard $n$-simplex
\begin{align*}
\Delta_n:=\Big{\{}(t_0,\ldots,t_n) \in \mathbb{R}^{n+1} \,|\,t_i \geq 0 \quad\text{for} \quad 1\leq i\leq n \quad \text{and} \quad \sum^n_{i=0}t_i=1\Big{\}}
\end{align*}
onto the space $\{x \in \mathbb{R}^{n+1} \,|\, x_0=0\}$. Then it is easily checked that the map
\begin{align*}
\varphi:D_n\rightarrow\Delta_n, \quad (0,t_1,\ldots,t_n)\mapsto\Big(1-\sum^n_{i=1}t_i,t_1,\ldots,t_n\Big)
\end{align*}
is a parametrization of $\Delta_n$ and therefore we have 
\begin{align*}
\int_{\Delta_n}FdO=\int_{D_n}(F\circ \varphi)\underbrace{\sqrt{\det({d\varphi}^T\cdot d\varphi)}}_{:=c_n}d\lambda_n
\end{align*}
for every smooth function $F\in C^{\infty}(\Delta_n)$. We continue with a Lie algebra extension
\begin{align}
0 \longrightarrow \mathfrak{n} {\longrightarrow}\widehat{\mathfrak{g}} \stackrel{q} {\longrightarrow} \mathfrak{g}\longrightarrow 0 \label{eq:LAE}
\end{align}
and a $\mathfrak{g}$-module $V$ which we also consider as a $\widehat{\mathfrak{g}}$-module with respect to the action $x.v:=q(x).v$ for $x\in\widehat{\mathfrak{g}}$ and $v\in V$. Furthermore, let $\sigma_0,\ldots,\sigma_n\in C^1(\mathfrak{g},\widehat{\mathfrak{g}})$ be linear sections of $q$. For each $t\in\Delta_n$ we define a linear section $\sigma_t \in C^1(\mathfrak{g},\widehat{\mathfrak{g}})$ of $q$ by
\begin{align*}
\sigma_t(x):=\sum^n_{i=0}t_i\cdot\sigma_i(x) \quad \forall x \in\mathfrak{g}
\end{align*}
and we write $R_t:=R_{\sigma_t} \in C^2(\mathfrak{g},\mathfrak{n})$ for the corresponding curvature. Given $p \in \mathbb{N}_0$ and a symmetric $p$-linear map $f\in\Sym^p(\mathfrak{n},V)^{\widehat{\mathfrak{g}}}$, we also define 
\begin{align*}
\Delta_f(\sigma_0,\ldots,\sigma_n):=\frac{1}{c_n}\cdot\int_{\Delta_n}f_{\alpha_1,\ldots,\alpha_n,R_t,\ldots,R_t} \, dO \in C^{2p-n}(\mathfrak{g},V),
\end{align*}
where $\alpha_i:=\sigma_i-\sigma_0\in C^1(\mathfrak{g},\mathfrak{n})$ and where 
\begin{align*}
f_{\alpha_1,\ldots,\alpha_n,R_t,\ldots,R_t} := \tilde{f}\circ(\alpha_1\wedge_{\otimes_s}\cdots\wedge_{\otimes_s}\alpha_n\wedge_{\otimes_s}R_t\wedge_{\otimes_s}\cdots\wedge_{\otimes_s} R_t)
\end{align*}
(\cf Section \ref{sec:pre+not}). In this way we obtain a map
\begin{align*}
\Delta(\sigma_0,\ldots,\sigma_n):\Sym^p(\mathfrak{n},V)^{\widehat{\mathfrak{g}}}\rightarrow C^{2p-n}(\mathfrak{g},V), \quad f\mapsto\Delta_f(\sigma_0,\ldots,\sigma_n)
\end{align*}
which extends to the so-called \emph{relative Bott--Lecomte homomorphism} 
\begin{align*}
\Delta(\sigma_0,\ldots,\sigma_n):\Sym^{\bullet}(\mathfrak{n},V)^{\widehat{\mathfrak{g}}}\rightarrow C^{\bullet}(\mathfrak{g},V).
\end{align*}
We point out that for $n=0$ this map is actually an algebra homomorphism. We are now ready to state and prove our main result which gives an important relation between the cochains $\Delta_f$.

\begin{thm}\label{thm:main thm}
For $f\in\Sym^k(\mathfrak{n},V)^{\widehat{\mathfrak{g}}}$ we have
\begin{align*}
(k-n+1)\cdot d(\Delta_f(\sigma_0,\ldots,\sigma_n))=\sum^n_{i=0}(-1)^i\Delta_f(\sigma_0,\ldots,\widehat{\sigma}_i,\ldots,\sigma_n),
\end{align*}
where $\widehat{\sigma}_i$ means that $\sigma_i$ is omitted.
\end{thm}

\begin{proof}
For the sake of legibility we divide the proof into several parts.
\begin{enumerate}
	\item
		For $t\in\Delta_n$ and $x\in\mathfrak{g}$ let $S_t(x):=\ad(\sigma_t(x))$. A short computation then yields
\begin{align*}
&(p-n+1)\cdot d_{\mathfrak{g}}(\Delta_f(\sigma_0,\ldots,\sigma_n))=\frac{p-n+1}{c_n}\cdot\int_{\Delta_n}d_{\mathfrak{g}}f_{\alpha_1,\ldots,\alpha_n,R_t,\ldots,R_t}dO
\\
&=\frac{p-n+1}{c_n}\cdot\int_{\Delta_n}d_{\mathfrak{g}}(\tilde{f}\circ(\alpha_1\wedge_{\otimes_s}\cdots\wedge_{\otimes_s}\alpha_n\wedge_{\otimes_s}R_t\wedge_{\otimes_s}\cdots\wedge_{\otimes_s} R_t))dO
\\
&=\frac{p-n+1}{c_n}\cdot\int_{\Delta_n}\tilde{f}\circ(d_{S_t}(\alpha_1\wedge_{\otimes_s}\cdots\wedge_{\otimes_s}\alpha_n\wedge_{\otimes_s}R_t\wedge_{\otimes_s}\cdots\wedge_{\otimes_s} R_t))dO.
\end{align*}
Since $d_{S_t}\alpha_i=d_{\mathfrak{g}}\alpha_i+[\alpha_i,\sigma_t]=\frac{d}{d t_i}R_t$, the abstract Bianchi identity $d_{S_t} R_t=0$ (\cf Lemma \ref{lem:cur+bianchi}) and Lemma  \ref{lem:cov der} imply that
\begin{align*}
&d_{S_t}(\alpha_1\wedge_{\otimes_s}\cdots\wedge_{\otimes_s}\alpha_n\wedge_{\otimes_s}R_t\wedge_{\otimes_s}\cdots\wedge_{\otimes_s} R_t)
\\
&=\sum^n_{i=1}(-1)^{i-1}\alpha_1\wedge_{\otimes_s}\cdots\wedge_{\otimes_s}\frac{d}{d t_i}R_t\wedge_{\otimes_s}\ldots\wedge_{\otimes_s}\alpha_n\wedge_{\otimes_s}R_t\wedge_{\otimes_s}\cdots\wedge_{\otimes_s} R_t.
\end{align*}
Therefore, we conclude that
\begin{align*}
(p-n+1)\cdot d_{\mathfrak{g}}(\Delta_f(\sigma_0,\ldots,\sigma_n))=\sum^n_{i=1}(-1)^{i-1}\frac{1}{c_n}\cdot\int_{\Delta_n}\frac{d}{d t_i}f_{\alpha_1,\ldots,\widehat{\alpha}_i,\ldots,\alpha_n,R_t,\ldots,R_t}dO.
\end{align*}
	\item
		Given $1 \leq i \leq n$ and $t=(t_0,\ldots,0,\ldots,t_n) \in \Delta^i_n$ (the $i$-th face of $\Delta_n$), consider
\begin{align*}
(\sigma_t)_i:=\sum_{j\in\{0,\ldots,\widehat{i},\ldots,n\}}t_j\cdot\sigma_j\in C^1(\mathfrak{g},\widehat{\mathfrak{g}})
\end{align*}		
and the corresponding curvature $(R_t)_i$. Furthermore, let $s_i:=\sum_{j\in\{1,\ldots,\widehat{i},\ldots,n\}}t_j$. Then a few moments thought show that $R_{t|t_i=0}=(R_t)_i$ and that $R_{t|t_i=1-s_i}=(R_t)_0$. Consequently we obtain, 
\begin{align*}
&\frac{1}{c_n}\cdot\int_{\Delta_n}\frac{d}{d t_i}f_{\alpha_1,\ldots,\widehat{\alpha}_i,\ldots,\alpha_n,R_t,\ldots,R_t}dO=\int_{D_n}\frac{d}{d t_i}f_{\alpha_1,\ldots,\widehat{\alpha}_i,\ldots,\alpha_n,R_t,\ldots,R_t}d\lambda_n
\\
&=\int^1_0\cdots\int^{1-s_i}_0\frac{d}{d t_i}f_{\alpha_1,\ldots,\widehat{\alpha}_i,\ldots,\alpha_n,R_t,\ldots,R_t}dt_i\cdot dt_1\ldots d\widehat{t}_i\ldots dt_n
\\
&=\int^1_0\ldots\int^{1-s_i+t_1}_0[f_{\alpha_1,\ldots,\widehat{\alpha}_i,\ldots,\alpha_n,R_t,\ldots,R_t}]^{1-s_i}_0dt_1\ldots d\widehat{t}_i\ldots dt_n
\\
&=\frac{1}{c_{n-1}}\cdot\int_{\Delta^i_n}f_{\alpha_1,\ldots,\widehat{\alpha}_i,\ldots,\alpha_n,(R_t)_0,\ldots,(R_t)_0}dO-\frac{1}{c_{n-1}}\cdot\int_{\Delta^i_n}f_{\alpha_1,\ldots,\widehat{\alpha}_i,\ldots,\alpha_n,(R_t)_i,\ldots,(R_t)_i}dO
\\
&=\frac{1}{c_{n-1}}\cdot\int_{\Delta^i_n}f_{\alpha_1,\ldots,\widehat{\alpha}_i,\ldots,\alpha_n,(R_t)_0,\ldots,(R_t)_0}dO-\Delta_f(\sigma_0,\ldots,\widehat{\sigma}_i,\ldots,\sigma_n).
\end{align*}
	\item
		Summing up our results so far, we have established the following equation:
\begin{align*}
&(p-n+1)\cdot d(\Delta_f(\sigma_0,\ldots,\sigma_n))
\\
&=\sum^n_{i=1}(-1)^{i-1}\frac{1}{c_{n-1}}\cdot\int_{\Delta^i_n}f_{\alpha_1,\ldots,\widehat{\alpha}_i,\ldots,\alpha_n,(R_t)_0,\ldots,(R_t)_0}dO+\sum^n_{i=1}(-1)^i\Delta_f(\sigma_0,\ldots,\widehat{\sigma}_i,\ldots,\sigma_n).
\end{align*}
It therefore remains to show that 
\begin{align*}
\Delta_f(\sigma_1,\ldots,\sigma_n)=\sum^n_{i=1}(-1)^{i-1}\frac{1}{c_{n-1}}\cdot\int_{\Delta^i_n}f_{\alpha_1,\ldots,\widehat{\alpha}_i,\ldots,\alpha_n,R_{t_0},\ldots,R_{t_0}}dO,
\shortintertext{where}
\Delta_f(\sigma_1,\ldots,\sigma_n)=\frac{1}{c_{n-1}}\cdot\int_{\underbrace{\Delta_{n-1}}_{\cong\Delta^i_n}}f_{\alpha_2-\alpha_1,\ldots,\alpha_n-\alpha_1,R_{t_0},\ldots,R_{t_0}}dO.
\end{align*}
	\item
		To verify the claim first recall that $f_{\alpha_2-\alpha_1,\ldots,\alpha_n-\alpha_1,R_{t_0},\ldots,R_{t_0}}$ is equal to
\begin{align*}
		\tilde{f}\circ((\alpha_2-\alpha_1)\wedge_{\otimes_s}\cdots\wedge_{\otimes_s}(\alpha_n-\alpha_1)\wedge_{\otimes_s}R_{t_0}\wedge_{\otimes_s}\cdots\wedge_{\otimes_s}R_{t_0}).
\end{align*}
Moreover, note that $\alpha\wedge_{\otimes_s}\alpha=0$ holds for every $\alpha\in C^1(\mathfrak{g},\mathfrak{n})$ by the symmetry of the wedge product $\wedge_{\otimes_s}$. A short calculation then shows
\begin{align*}
&(\alpha_2-\alpha_1)\wedge_{\otimes_s}\cdots\wedge_{\otimes_s}(\alpha_n-\alpha_1)\wedge_{\otimes_s}R_{t_0}\wedge_{\otimes_s}\cdots\wedge_{\otimes_s}R_{t_0}
\\
&=\sum^n_{i=1}(-1)^{i+1}\alpha_1\wedge_{\otimes_s}\cdots\wedge_{\otimes_s}\widehat{\alpha}_i\wedge_{\otimes_s}\ldots\wedge_{\otimes_s}\alpha_n\wedge_{\otimes_s}R_{t_0},\cdots\wedge_{\otimes_s}R_{t_0}
\end{align*}
which in turn yields the desired formula
\begin{align*}
f_{\alpha_2-\alpha_1,\ldots,\alpha_n-\alpha_1,R_{t_0},\ldots,R_{t_0}}=\sum^n_{i=1}(-1)^{i-1}f_{\alpha_1,\ldots,\widehat{\alpha}_i,\ldots,\alpha_n,R_{t_0},\ldots,R_{t_0}}
\end{align*}
and therefore finally completes the proof.
\qedhere
\end{enumerate}
\end{proof}

Applying Theorem \ref{thm:main thm} in the cases $n=0,1$ yields the following corollaries:

\begin{cor}\label{cor1}
For $n=0$ we have $d_{\mathfrak{g}}(\Delta_f(\sigma))=0$. Hence, the cochain $\Delta_f(\sigma)$ is a cocycle and defines an element of $H^{2p}(\mathfrak{g},V)$. Furthermore, for each $p \in \mathbb{N}_0$ the relative Bott--Lecomte homomorphism induces a map
\begin{align*}
[\Delta(\sigma)]:\Sym^p(\mathfrak{n},V)^{\widehat{\mathfrak{g}}}\rightarrow H^{2p}(\mathfrak{g},V), \quad f\mapsto [\Delta_f(\sigma)].
\end{align*}
\end{cor}

\begin{cor}\label{cor2}
For $n=1$ and $p \in \mathbb{N}_0$, we have 
\begin{align*}
p\cdot d_{\mathfrak{g}}(\Delta_f(\sigma_0,\sigma_1))=\Delta_f(\sigma_0)-\Delta_f(\sigma_1)\in C^{2p}(\mathfrak{g},V).
\end{align*}
In particular, we have $[\Delta_f(\sigma_0)]=[\Delta_f(\sigma_1)]\in H^{2p}(\mathfrak{g},V)$ and hence the map $[\Delta(\sigma)]$ of Corollary \ref{cor1} is independent of the connection $\sigma_0$ and is, in fact, the homomorphism described in Theorem \ref{thm:lecomte} .
\end{cor}

With Corollary \ref{cor2} available we are now ready to define secondary characteristic classes in the following way: For linear sections $\sigma,\sigma_1,\sigma_2\in C^1(\mathfrak{g},\widehat{\mathfrak{g}})$ of $q$
first consider the spaces
\begin{align*}
\Sym^p(\mathfrak{n},V)^{\widehat{\mathfrak{g}}}_{(\sigma)}:=\{f\in\Sym^p(\mathfrak{n},V)^{\widehat{\mathfrak{g}}} \,|\,f_{\sigma}=0\}
\shortintertext{and}
\Sym^p(\mathfrak{n},V)^{\widehat{\mathfrak{g}}}_{(\sigma_1,\sigma_2)}:=\Sym^p(\mathfrak{n},V)^{\widehat{\mathfrak{g}}}_{(\sigma_1)}\cap \Sym^p(\mathfrak{n},V)^{\widehat{\mathfrak{g}}}_{(\sigma_2)}.
\end{align*}
Given $f\in\Sym^p(\mathfrak{n},V)^{\widehat{\mathfrak{g}}}_{(\sigma_1,\sigma_2)}$, it then follows from Corollary \ref{cor2} that $d_{\mathfrak{g}}(\Delta_f(\sigma_0,\sigma_1))=0$. Therefore, the cochain $\Delta_f(\sigma_0,\sigma_1)$ actually defines a cohomology class in $H^{2p-1}(\mathfrak{g},V)$.

\begin{defn}
We call $[\Delta_f(\sigma_0,\sigma_1)]\in H^{2p-1}(\mathfrak{g},V)$ the \emph{secondary characteristic class} of the Lie algebra extension in Equation (\ref{eq:LAE}) associated to the triple $(\sigma_1,\sigma_2,f)$.
\end{defn}

Note that secondary characteristic classes may be used to associate characteristic classes to split Lie algebra extensions, that is, to semidirect product constructions, because in this situation the primary characteristic classes vanish. We therefore conclude with following example. 

\begin{expl}
For the convenience of the reader we present the example divided into several parts.
\begin{enumerate}
	\item
		Let $\mathfrak{h}_3(\R)$ be the three-dimensional Heisenberg algebra, that is, the three-dimensional vector space with the basis $p,q,z$ equipped with the skew-symmetric bracket determined by
$[p,q]=z$ and $[p,z]=[q,z]=0$. It is easily checked that the linear endomorphism $D$ of $\mathfrak{h}_3(\R)$ defined by $Dz=0$, $Dp=q$ and $Dq=-p$ is actually a derivation of $\mathfrak{h}_3(\R)$, so that we obtain a Lie algebra $\widehat{\mathfrak{g}}:=\mathfrak{h}_3(\R) \rtimes_D \R$ which is called the \emph{oscillator algebra}. 
	\item
		Consider $V:=\R$ as a trivial $\R$-module. Then it is well-known that the corresponding first Lie algebra cohomology $H^1(\R,\R)$ is isomorphic to $\Hom_{\text{Liealg}}(\R,\R) \cong\R$.
	\item
		Define sections $\sigma_0$ and $\sigma_z$ of the projection map $q:\widehat{\mathfrak{g}} \to \R$ by $\sigma_0(r):=(0,r)$ and $\sigma_z(r):=(z,r)$, respectively. Then a short computation yields $R_{\sigma_0}=R_{\sigma_z}=0$. Moreover, a few moments thought show that
\begin{align*}
\Sym^1(\mathfrak{h}_3(\R),\R)^{\widehat{\mathfrak{g}}}_{(\sigma_0)}
=\Sym^1(\mathfrak{h}_3(\R),\R)^{\widehat{\mathfrak{g}}}_{(\sigma_z)}
=\R \cdot f_z,
\end{align*}
where $f_z:\mathfrak{h}_3(\R) \to \R$ denotes the linear functional determined by $f_z(z)=1$ and $f_z(p)=f_z(q)=0$. Consequently, $$\Sym^1(\mathfrak{h}_3(\R),\R)^{\widehat{\mathfrak{g}}}_{(\sigma_0,\sigma_z)}=\R \cdot f_z.$$
	\item
		Finally, the secondary characteristic class of the oscillator algebra $\mathfrak{h}_3(\R) \rtimes_D \R$ associated to the triple $(\sigma_0,\sigma_z,f_z)$ is given by
\begin{align*}
\Delta_{f_z}(\sigma_0,\sigma_z)=\frac{1}{c_1}\cdot\int_{\Delta_1}{(f_z)}_{\sigma_z-\sigma_0} \, dO = \int_{[0,1]} 1 \,d\lambda_1=1 \in H^1(\R,\R),
\end{align*}
because $f_z ((\sigma_z-\sigma_0)(r))=f_z(z)=1$ holds for all $r \in \R$.
\end{enumerate}
\end{expl}

\section*{Acknowledgement}

I would like to express my greatest gratitude to my former advisor K.-H. Neeb for introducing me to fibre bundles and for stimulating discussions on the subject matter of this paper. I would also like to thank the \emph{Studienstiftung des deutschen Volkes} and \emph{Carl Tryggers Stiftelse f\"or Vetenskaplig Forskning} for supporting this research.

\end{document}